\documentclass[reqno]{amsart}
\usepackage{hyperref}
\usepackage{graphicx,color}
\usepackage[section]{algorithm}
\usepackage{algorithmic}
\usepackage{enumerate}
\usepackage{mathrsfs,enumerate}
\usepackage{multirow,array}
\usepackage{subfigure}
\usepackage{pstricks,pst-node,pst-tree, pst-plot, pst-eps}
\usepackage{xspace}
\usepackage{marginnote}
\usepackage{stmaryrd}
\usepackage{upgreek}
\usepackage{bm}
\usepackage{graphicx}

\theoremstyle{plain}
\newtheorem{theorem}{Theorem}[section]
\newtheorem{lemma}[theorem]{Lemma}

\theoremstyle{theorem}

\newtheorem{remark}{Remark}[section]

\newtheorem{assumption}{Assumption}[section]

\graphicspath{{figures/}}

\def\cQ{{\mathcal Q}}
\def\CC{{\mathbb C}}
\def\cS{\mathcal S}
\def\cA{\mathcal A}
\def\beq#1\eeq{\begin{equation} #1 \end{equation}}
\def\bal#1\eal{\begin{aligned} #1 \end{aligned}}
\def\Forall{\qquad \hbox{for all }}

\def\blue#1{#1}
\newcommand{\vh}{\mathbb V_h}
\newcommand{\cT}{\mathcal T}
\newcommand{\HH}{\mathbb H}
\newcommand{\VV}{\mathbb V}
\newcommand{\RR}{\mathbb R}
\newcommand{\RRD}{\mathbb{R}^d}

\renewcommand{\Re}{\mathfrak{Re}}
\renewcommand{\Im}{\mathfrak{Im}}
\begin{document}

\title[Sinc Approximations of Fractional Powers]
{On Sinc Quadrature Approximations of Fractional Powers of Regularly Accretive Operators}

\author[A.~Bonito]{Andrea Bonito}
\address[A.~Bonito]{Department of Mathematics, Texas A\&M University, College Station, TX 77843, USA}
\thanks{AB is partially supported by NSF grant DMS-1254618.}
\email{bonito@math.tamu.edu}

\author[W.~Lei]{Wenyu Lei}
\address[W.~Lei]{Department of Mathematics, Texas A\&M University, College Station, TX 77843, USA}
\thanks{WL is partially supported by NSF grant DMS-1254618.}
\email{wenyu@math.tamu.edu}

\author[J.~Pasciak]{Joseph E. Pasciak}
\address[J.~Pasciak]{Department of Mathematics, Texas A\&M University, College Station, TX 77843, USA}
\email{pasciak@math.tamu.edu}

\date{\today}

\begin{abstract} 
We consider the finite element approximation of fractional powers of regularly accretive operators via the Dunford-Taylor integral approach. We use a sinc quadrature scheme to approximate the Balakrishnan representation of the negative powers of the operator as well as its finite element approximation. We improve the exponentially convergent error estimates from [A. Bonito, J. E. Pasciak, IMA J. Numer. Anal. (2016) 00, 1-29] by reducing the regularity required on the data. Numerical experiments illustrating the new theory are provided.
\end{abstract} 


\maketitle

\section{Introduction.}
Let $Y\subset X$  be complex valued Hilbert spaces with $Y$ dense in $X$. 
Let $A(\cdot,\cdot)$ be a bounded and coercive sesquilinear form on
$Y$.   Following \cite{kato1961} (see, Section~\ref{s:reg-acc}) there is
a uniquely defined unbounded operator on $X$ denoted by $A$ with domain
$D(A)\subset Y$.   The fractional powers of $A$ are given by the
Balakrishnan
 integral
 \cite{balakrishnan,kato1961} defined for $\beta \in (0,1)$ and $f\in X$:
\begin{equation}\label{e:negative-power}
	u=A^{-\beta}f = \frac{\sin(\pi \beta)}{\pi} \int_0^\infty \mu^{-\beta} (\mu I+ A)^{-1}f\, d\mu .
\end{equation}
After the  change of variable $\mu = e^y$, we approximate the resulting integral using a truncated equally spaced quadrature (sinc quadrature \cite{lundbowers})
given by 
\beq
u_{k} := \cQ^{-\beta}_k(A) f := \frac{k\sin(\pi \beta)}{\pi}      \sum_{\ell=-M}^N e^{(1-\beta) y_\ell} (e^{y_\ell}I+  A)^{-1} f,
\label{sinc-abs}
\eeq
where $k>0$ is a real number and $N$ and $M$ are positive integers
chosen to be  on the order of $1/k^2$ (see Remark~\ref{rem:MN}).

The fractional powers $A^s$ (with domain $D(A^s)$) 
are well defined for $s\ge 0$, see \cite{lunardi}.   It is a consequence of
the coercivity of $A(\cdot,\cdot)$ that the natural norm on $D(A^s)$,
namely $(\|v\|_X^2+\|A^s v\|_X)^{1/2}$ is
equivalent to the norm 
$$\|v\|_{D(A^s)}:= \|A^s v \|_X.$$

The first result of this note, Theorem~\ref{l:exp} shows that the
quadrature error goes exponentially to zero when measured
in a scale of the above norms.
Specifically,  for any
$s\in [0,\beta)$ and $t\ge 0$,
\beq
\| u - \cQ^\beta_k(A) f \|_{D(A^{s+t})} \le  C \,e^{-c/k}  \|f
\|_{D(A^t) }
\label{new}
\eeq
with $C$ and $c$ not depending on $k$.  
The result of this paper requires less regularity on $f$ than 
 that of \cite{BP16} which is essentially of the form (for $s\ge 0$) 
\beq\| (u - \cQ^\beta_k(A) f) \|_{D(A^s)} \le C\, e^{-c/k}  \|
f\|_{D(A^{s}) }.
\label{old}
\eeq

In the remainder of the paper, we focus on \blue{fractional problems
  with   homogeneous 
Dirichlet boundary conditions.  We avoid the non-homogeneous case (see 
 \cite{APR17} for 
a discussion of one possible definition for fractional problems
with 
non-homogeneous Dirichlet and Neumann boundary conditions)}.  
Let $\Omega\subset \RRD$ be a bounded domain with Lipschitz boundary
$\Gamma:=\partial \Omega$.  We assume that there are two 
 open sets (with respect to Lebesgue measure in
$\RR^{d-1}$)
$\Gamma_D$ (Dirichlet) and $\Gamma_N$ (Neumann) such that $\Gamma_D\cap \Gamma_N = \emptyset$ and
$\overline{\Gamma_D\cup \Gamma_N} =\Gamma$. We additionally assume that
$\Gamma_D\neq \emptyset$.
We define the complexed valued functional space
$$\VV := \{v\in H^1(\Omega): v=0 \text{ on }\Gamma_D\} \subset H^1(\Omega)$$
and set
\begin{equation}\label{example_space}
X:=L^2(\Omega), \qquad \textrm{and} \qquad Y:=\VV.
\end{equation}
Also for all $u,v \in \VV$ we consider
\begin{equation}\label{example}
	A(u,v):=\int_\Omega \mathbf A \nabla u \cdot \nabla\overline{v}
        + \mathbf b_1\cdot\nabla u~\overline{v} + u ~\mathbf
        b_2\cdot\nabla\overline{v} +c u~\overline{v}\, dx,
\end{equation}
where $\overline{v}$ denotes the complex conjugate of $v$.
The coefficients $\mathbf A\in L^\infty(\Omega,\mathsf{GL}(\RRD))$,
$\mathbf b_1,\mathbf b_2\in L^\infty(\Omega,\RRD)$, $c\in
L^\infty(\Omega)$ are assumed to be such that the form $A(\cdot,\cdot)$
is coercive and bounded on $\VV$.

The numerical  approximation of \eqref{e:negative-power} with
$X,Y,A(\cdot,\cdot)$ and $\VV$ as in  \eqref{example_space}-\eqref{example}.
is defined as follows.   One starts by introducing a finite element
 space $\vh$.   The general framework of Section~\ref{s:reg-acc} is used
 to define the discrete fractional power $A_h^\beta$  for
 $X=Y=\vh$ and $A(\cdot,\cdot)$ as in \eqref{example}.
The (semi-discrete) finite element approximation of $u=A^{-\beta}f$ is then defined by
$u_h= A_h^{-\beta} \pi_h f$ where $\pi_h $ denotes the $L^2(\Omega)$
projection onto $\vh$.  Of course, 
\beq
u_h = A_{h}^{-\beta} \pi_h f =  \frac{\sin(\pi \beta)}{\pi} \int_0^\infty \mu^{-s} (\mu I+ A_h)^{-1}\pi_h f\, d\mu.
\label{discah}
\eeq
The fully discrete approximation is then defined by $u_{h,k}=
\cQ_k^{-\beta} (A_h)\pi_h f$ and our goal is to estimate the error
$u-u_{h,k}$.

We shall provide error estimates in Sobolev norms.  To this end, we
define 
\beq
	\HH^r(\Omega) 
:= \left\{
	\begin{aligned}
		&(L^2(\Omega), H^1_0(\Omega))_{2,r}, &\qquad\text{for } r\in[0,1],\\
		&H^r(\Omega)\cap H^1_0(\Omega), &\qquad\text{for } r\in[1,2],
	\end{aligned}
	\right.
\label{hhdef}
\eeq
        with $(\cdot,\cdot)_{2,r}$ denoting interpolation using the real method.
We assume elliptic regularity with index $\alpha\in (0,1]$ (see,
Assumption~\ref{a:elliptic-regularity}).    When this condition holds, 
\beq
D(A^{s/2})=\HH^s(\Omega),\quad \hbox{ for } s\in [0,1+\alpha],
\label{e:hh-Ar}
\eeq
and their norms are equivalent.
Indeed, for $s\in [0,1)$, this follows from Theorem 3.1 of
\cite{kato1961}, $s=1$ is proved in \cite{agranovich} and $s\in
(1,1+\alpha]$ is given by  Theorem 6.4 of \cite{BP16}.  

We shall not assume
artificial smoothness of the solution $u=A^{-\beta}f$.  Instead, we assume smoothness on
$f$ and use the smoothing properties of the operator $A^{-\beta}$ and \eqref{e:hh-Ar}
to
conclude regularity for $u$.    
Note that for $t\ge 0$, $u$ is in $D(A^{t+\beta}) $ if and only if $f$
is in $D(A^t)$ with equal norms. 
In all cases,  we assume that $f$ is 
 in $X:=L^2(\Omega)$.

        Estimates for the  error between $u$ and $u_h$ in the
        $\HH^r(\Omega)$ norms 
       for $r\in [0,1]$
were provided in
\cite{BP16} and are also discussed in Subsection~\ref{ss:uuhk}.   We shall apply
Theorem~\ref{l:exp} to the derive bounds for the error
$\|u_h-u_{h,k}\|_{\HH^r(\Omega)}$
and obtain (Theorem~\ref{l:expex})
\beq
\|u_h-u_{h,k}\|_{\HH^r(\Omega)}\preceq
e^{-\blue{\pi^2/(2k)}}\|f\|_{\HH^{\max(r+2\alpha^*-2\beta,0)}(\Omega)}
\label{expconv}
\eeq
with $2\alpha^*=\alpha+\min(\alpha,1-r)$.
Note that the norm on $f$ appearing above is always controlled by the
norm on $f$ needed in \cite{BP16} to obtain an $h^{\epsilon}$ ($0<\epsilon \ll 1$)
convergence bound for $\|u-u_h\|_{\HH^r(\Omega)}$ (see
Subsection~\ref{ss:uuhk}) so the exponential
  convergence of the sinc approximation  is achieved without
 additional assumptions on $f$.

This improves the estimates obtained for the sinc approximation in \cite{BP16}.
For example, if $\Omega$ is convex and the coefficients of $A(\cdot,\cdot)$ are
smooth,  $D(A^s)=\HH^{2 s}(\Omega)$ and the norms $\| A^s\cdot \|$ and
$\|\cdot\|_{\HH^{2s}(\Omega)}$ are equivalent for $s\in [0,1]$. 
As shown in \cite{BP16}, the energy norm error, $\|u-u_h\|_{H^\beta(\Omega)}$ is $O(h^{2-\beta})$
up to a logarithm of $h^{-1}$ 
for $f\in \HH^{2-2\beta}(\Omega)$.   In this case, \eqref{expconv} 
implies  exponential
convergence of the sinc approximation in the energy norm while \eqref{old}
requires $f\in \HH^\beta(\Omega) $ which corresponds to  more regularity
 when $\beta > 2/3$.

We refer to \cite{bonito2017numerical2} for a review of different numerical methods tailored to fractional diffusion.
To the best of our knowledge, besides the above mentioned works, there are no alternative numerical method for the approximation of fractional powers of general regularly accretive operators. 
However, several methods are available when the operator $A$ is real symmetric. 
We refer to \cite{MR2300467,MR2252038,MR2800568} for methods based on expansions using the  eigenpairs of the  discretized operator and to \cite{NOS15} as well as to \cite{meidner2017hp,banjai2017tensor} where approximations of the  ``Neumann to Dirichlet'' map of an extended problem is advocated.
In \cite{AB,d2013fractional}, numerical approximation of the integral definition of the fractional laplacian is considered.
It is worth mentioning that in this context, the recent work \cite{bonito2017numerical} is also based on sinc quadrature approximations of a Dunford-Taylor representation.

The outline of this note is as follows. In Section~\ref{s:reg-acc}, we
review the definition and properties of fractional powers
of the regularly accretive operator $A$.
In Section~\ref{s:SINC}, we prove abstract estimates  for the error in
the sinc quadrature
error $u_h-u_{h,k}$ showing  exponential convergence in the norms of $D(A^s)$, for
$s\ge 0$ under appropriate regularity conditions on $f$. 
We consider the setting described by \eqref{example_space}-\eqref{example} in Section~\ref{app-example} and 
provide error estimates for $u-u_{h,k}$ by combining
the estimates for the error $u-u_h$ given by \cite{BP16} and the
results of Section~\ref{s:SINC}.  We provide numerical illustrations of
the improved theory in Section~\ref{s:numerical}.

We write $a \preceq b$ to mean $a \leq C b$, with a constant $C$ that
does not depend on $a$, $b$, or  the discretization parameters.   
Finally, $a \approx b$ indicates $a\preceq b$ and $b\preceq a$.\\

\paragraph{\textsc{Acknowledgment}} 

The authors would like to thank R.H.~Nochetto for pointing out the possible sub-optimality in \cite{BP16}, thereby prompting the current analysis.

\section{Fractional Powers of Regularly Accretive Operators.}\label{s:reg-acc}

\def\VV{Y}
As in the introduction, we consider complex valued Hilbert spaces with
$Y$ continuously and densely imbedded in $X$  and  a bounded and coercive
sesquilinear form
$A(\cdot,\cdot)$.    This means that there are constants $c_0$ and $C_0$
satisfying
\begin{equation}\label{i:form-strongly-elliptic}
	\Re(A(v,v))\geq c_0 \|v\|_{\VV}^2, \Forall v\in \VV,
\end{equation}
and
\begin{equation}\label{i:form-bounded}
	|A(u,v)|\leq C_0 \|u\|_{\VV} \|v\|_{\VV}, \Forall u,v\in\VV .
\end{equation}
Such a  sesquilinear form $A(\cdot,\cdot)$ is called regular
(cf. \cite{kato1961}).  By possibly rescaling the norm in $X$, we may
assume that 
\beq\|y\|_X\le \|y\|_Y,\Forall y\in Y.
\label{rescaleX}
\eeq
To simplify the notation, we denote $\|v\|:=\|v\|_X$ for $v\in X$ and
$\| \cA\| $ to be the operator norm of $\cA$ when $\cA$ is a bounded
operator from $X$ into $X$.

Following \cite{kato1961},  there is
a uniquely defined unbounded operator on $X$ denoted by $A$ with domain
$D(A)\subset Y$ defined as follows.   Invoking the Lax-Milgram Theorem, we define the one-to-one solution operator $T : X\to \VV$ satisfying 
\beq
	A(Tf , \phi) = (f,\phi),\Forall \phi \in \VV .
        \label{Tdef}
\eeq
The unbounded operator $A$ is defined by  $Aw:=T^{-1}w$ for $w$ in  
$D(A):=\text{Range}(T)$.   The operator $A$ is a
closed densely defined operator on $X$ with domain $D(A)$.  
Such an operator  is said to be regularly accretive (cf. \cite{kato1961}).

For regularly accretive operators, there exists $\omega\in[0,\pi/2)$
such that the spectrum of $A$
 is contained in the sector
$S_\omega:=\{z\in \CC : |\arg z|\leq \omega\}$. 
We also have the following estimate for the resolvent $R_z(A):=(A-zI)^{-1}$
\begin{equation}\label{i:resolvent-L2}
	\|R_z(A)f\|\le (\sin(\pi/2-\omega))^{-1} |z|^{-1}\|f\|,\qquad
        \text{for } \Re(z)<0.
\end{equation}
Also, if $z$ is negative,  $ \|R_z(A)f\|\le
|z|^{-1}\|f\|$, i.e., $A$ is M-accretive. The above bounds are 
a consequence of Theorem  2.2 of \cite{kato1961}.
For later use, we also note that the coercivity assumption
\eqref{i:form-strongly-elliptic} and \eqref{rescaleX} implies that for $\Re(z)\le c_0/2$,
\begin{equation}\label{i:resolvent-H1}
\|R_z(A)f\|\le \|R_z(A)f\|_{Y}\le \frac{2}{c_0} \|f\|.
\end{equation}

\subsection{Fractional Powers} \label{ss:frac pow}
For $\beta\in(0,1)$, the negative fractional powers $A^{-\beta}$ of a regularly accretive operator $A$ are defined by  \eqref{e:negative-power}. It follows from
\eqref{i:resolvent-L2} and \eqref{i:resolvent-H1} that the integral \eqref{e:negative-power} 
is Bochner integrable in $X$ and so $A^{-r}$ is a bounded operator on $X$.
When $r$ is non-negative and not an integer with $n-1<r<n$, then 
$$D(A^r):=\{x\in X\ : \ A^{r-n}x \in D(A^n)\}, \qquad A^rx = A^n
A^{r-n}x.$$

We shall make use of the following commutivity proproperties involving fractional powers and the
resolvent, see e.g. \cite{lunardi}.  
\begin{enumerate} [(a)]
\item For $r\ge 0 $ and $\beta\in (0,1)$,
\beq
A^{-\beta} A^r x= A^{r-\beta}x = A^r A^{-\beta}x,\Forall x\in D(A^r).
\label{e:commutative1}
\eeq
\item For $z\in \rho(A)$ and $r\ge0$,  $R_z(A):D(A^r)\rightarrow
  D(A^{r+1})$ and 
\begin{equation}\label{e:commutative}
A^r R_z(A) x = R_z(A) A^rx, \Forall x\in D(A^r).
\end{equation} 
\end{enumerate}

\subsection{Interpolation Scales}

Since a regularly accretive operator is M-accretive, Corollary 4.3.6 of
\cite{lunardi} shows that  for $s\in (0,1)$,
$$[X,D(A)]_s = D(A^s)$$
with $[\cdot,\cdot]_s$ denoting the interpolation scale using the
complex method.  Corollary 2.1.8 of \cite{lunardi} then implies that
\beq
\|A^s v\| \preceq \|v\|_{[L^2(\Omega),D(A)]_s} \preceq \|Av\|^s
\|v\|^{1-s},\Forall v\in D(A).
\label{scaleineq}
\eeq

\section{Sinc Approximations to  $A^{-\beta}$.}\label{s:SINC}

In this section, we revisit the sinc approximation technique developed
in \cite{BP16} (see, \eqref{sinc-abs}) and provide an abstract theorem which
weakens  
 the regularity required on $f$ to achieve an exponential rate of convergence. 

As anticipated in the introduction, we use the change of variable $\mu:=e^y$ in
\eqref{e:negative-power}, i.e., 
\beq\label{e:sinc-int}
A^{-\beta}= \frac{\sin(\pi \beta)}{\pi} \int_{-\infty}^\infty e^{(1-\beta)y} (e^yI+A)^{-1}\, dy.
\eeq
For any positive integers $N$ and $M$ and a positive quadrature step $k$, we define 
the sinc approximation of $A^{-\beta}$ by \eqref{sinc-abs} which
corresponds to a truncated equally spaced quadrature approximation to \eqref{e:sinc-int}.
Notice that to simplify the notation, we do not specify the dependency on $M$ and $N$. In any event, in practice both $M$ and $N$ are functions of $k$ (see Remark~\ref{rem:MN}).

\subsection{Error Analysis}

The analysis for the sinc approximation error involves the analyticity
and decay properties of the function
$$
	F(z;\theta,\eta):=  e^{(1-\beta)z} (A^s(e^z I+A)^{-1}\theta,\eta).
$$
on the band
$$
z\in \cS_{\pi/2}:= \{ z \in \CC \ : \ |\Im(z)| \leq \pi/2 \}.
$$
Here
$(\cdot,\cdot)$ denotes the inner product on $X$ and $\theta,\eta$ are
fixed in $X$.
The decay properties are addressed in the following lemma.

\begin{lemma}[Integrand Estimate]\label{l:decay}
Let $s$ be in $[0,\beta)$ and $\delta$ be non-negative.
For $z \in \cS_{\pi/2}$,
\begin{equation}\label{i:decay}
	\| e^{(1-\beta)z} A^{s}(e^z I+A)^{-1} \| \preceq  \left\lbrace
	\begin{array}{ll}
	e^{(s-\beta)\Re(z)}&:\quad  \hbox{when } \Re(z)>0,\\
	e^{(1-\beta)\Re(z)} &:\quad  \hbox{when } \Re(z)\le 0. 
	\end{array}\right. 
\end{equation}
The hidden constant depends only on $\beta$, $s$, $c_0$ and $C_0$.
\end{lemma}
\begin{proof}
We fix $z \in \cS_{\pi/2}$ and $\theta\in X$.  As $-e^z$ is in $\rho(A)$ (in fact, $\Re(-e^z)<0$),
$(e^zI+A)^{-1} \theta $ is in $D(A)$.
Applying \eqref{scaleineq} 
gives
\begin{equation*}
	\| A^{s}(e^z I+A)^{-1}\theta \| \preceq \|(e^z I+A)^{-1}\theta\|^{1-s} \|A(e^z I+A)^{-1}\theta\|^s.
\end{equation*}
We apply \eqref{i:resolvent-L2} and \eqref{i:resolvent-H1} to obtain 
\begin{equation}\label{i:decay2}
 \|(e^zI+A)^{-1}\theta \|  \le \|\theta \|\left\lbrace
 \begin{array}{ll}
(\sin(\pi/2-\omega))^{-1}  e^{-\Re(z)},&\quad  \textrm{when } \Re(z)>0,\\
c_0/2,  & \quad \textrm{when } \Re(z) \leq 0.
\end{array}\right.
\end{equation}
Also \eqref{i:resolvent-L2} implies that
$$\|e^z (e^zI+A)^{-1}\| \le  (\sin(\pi/2-\omega))^{-1}, \Forall z\in \cS_{\pi/2}.$$
and hence 
$$
\|A(e^zI+A)^{-1}\| = \|I-e^z (e^zI+A)^{-1}\|\le  1+(\sin(\pi/2-\omega))^{-1}.
$$
Combining the above estimates gives
\begin{equation}\label{i:decay1}
	\| A^{s}(e^z I+A)^{-1} \| \preceq  \left\lbrace
	\begin{array}{ll}
	e^{(s-1)\Re(z)}&:\quad  \hbox{when } \Re(z)>0,\\
	1 &:\quad  \hbox{when } \Re(z)\le 0. 
	\end{array}\right. 
\end{equation}
The lemma follows from the above estimates and the trivial estimate,
$$|e^{(1-\beta)z}| = e^{(1-\beta) \Re(z)}.
$$
\end{proof}

Now by \eqref{i:decay2} and \eqref{i:decay1}, for $z\in \cS_{\pi/2}$,
$$ 
\begin{aligned}
|\frac{d}{dz} F(z;\theta,\eta)| \leq |e^{(2-\beta)z} (A^s(e^z
I+A)^{-2}\theta,\eta)| + (1-\beta)|F(z;\theta,\eta)|
	\\\preceq \|\theta\|\|\eta\|\left \{
	\begin{array}{ll}
	e^{(1+s-\beta)\Re(z)}&:\quad  \hbox{when } \Re(z)>0,\\
	e^{(1-\beta)\Re(z)} &:\quad  \hbox{when } \Re(z)\le 0, 
	\end{array}\right .
\end{aligned}
$$
i.e., $F(.;\theta,\eta)$ is analytic on $\cS_{\pi/2}$ for each $\theta$
and $\eta$ in $X$.        

We are now in position to prove our result which bounds the 
quadrature error.
  
\begin{theorem}[Sinc Quadrature Error]   \label{l:exp}
For $M, N>0$  integers, 
and $k >0$, let $\cQ_k^{-\beta}(A)$ be defined by \eqref{sinc-abs}.
Given $f\in D(A^t) $ with $t\ge 0$ and $-t\le s<\beta$,
then 
\begin{equation}
\begin{split}
&\|(A^{-\beta} - \cQ^{-\beta}_k(A)) f\|_{D(A^{s+t})} \preceq 
\bigg[
 (\sinh(\blue{\pi^2/(2k)}))^{-1}e^{-\blue{\pi^2/(2k)}} \\
& \qquad\qquad +  e^{-(\beta-s^+)N k} +  e^{-(1-\beta) M k}\bigg]\|f\|_{D(A^t)} ,
\end{split}
\label{aquad}
\end{equation}
where $s^+=\max(0,s)$ and the hidden constant depends only on $\beta$, $s$, $c_0$ and $C_0$.
\end{theorem}
\begin{proof}
Using \eqref{e:commutative}, \eqref{e:commutative1} 
 and the definition of the norm in \eqref{aquad}, \eqref{aquad} can be rewritten
$$
\|(A^{-\beta} - \cQ^{-\beta}_k(A))A^{s+t} f\| \preceq C(k) \|A^t
f\|,\Forall f\in D(A^t)$$
with $C(k)$ denoting the expression in brackets on the right hand side of \eqref{aquad}.
Since $A^t$ is a one to one map of $D(A^t)$ onto $X$ and $A^s$ is a bounded operator when $s<0$,
it is suffices to show that
\beq  |((A^{-\beta} - \cQ^{-\beta}_k(A))A^{s^+}\theta,\eta)| \preceq
C(k),\Forall \theta,\eta\in X
\label{to-show}
\eeq
with  $\|\theta\|=\|\eta\|=1$.

Extending  the sum in \eqref{to-show} to a sum over all integers and applying the triangle
inequality gives
\begin{equation}\bal
  |((A^{-\beta} - \cQ^{-\beta}_k(A))A^{s^+}\theta,\eta)|&\le
  \bigg| \int_{-\infty}^\infty F(y;\theta,\eta)dy - k \sum_{\ell =
    -M}^N F(\ell k;\theta,\eta)\bigg|\\
  &\qquad +
\bigg|k \sum_{\ell > N} F(\ell k;\theta,\eta)\bigg|
+\bigg|k \sum_{\ell <-M } F(\ell k;\theta,\eta)\bigg|.
\eal
\label{triple}
\end{equation}

\boxed{1}
For the second sum  on the right hand side of \eqref{triple},
applying Lemma~\ref{l:decay} gives
$$
k\sum_{\ell>N} | F(\ell k;\theta,\eta) | \preceq \int_{ N k}^\infty
e^{(s^+-\beta)y} \,  dy = e^{(s^+-\beta)N k}/(\beta-s^+).
$$
Similarly,
$$
k\sum_{\ell<-M} | F(\ell k;\theta,\eta) | \preceq \int_{-\infty}^{-Mk}
e^{(1-\beta)y} \, dy
= e^{(\beta-1)M k}/(1-\beta).
$$

\boxed{2} For first sum on the right hand side of \eqref{triple},
we invoke Theorem 2.20 in \cite{lundbowers}, which states that
\begin{equation}\label{e:result_sinc_thm}
\left| \int_{-\infty}^\infty F(y;\theta,\eta) - k \sum_{\ell = -\infty}^\infty F(\ell k;\theta,\eta)\right| \leq \frac{R}{2\sinh(\blue{\pi^2/(2k)})} e^{-\blue{\pi^2/(2k)}}
\end{equation}
provided that 
\begin{equation}\label{e:cond1}
\int_{-\pi/2}^{\pi/2} |F(t+iy;\theta,\eta)| dy \preceq 1, \qquad \Forall t \in \mathbb R
\end{equation}
and $R$ is a constant such that
\begin{equation}\label{e:cond2}
\int_{-\infty}^\infty \left( |F(y-i\pi/2;\theta,\eta)| + |F(y+i \pi/2; \theta,\eta) |\right) dy \leq R.
\end{equation}

Both conditions follows from Lemma~\ref{l:decay}.
In particular,
$$R=2C\left((\beta-s^+)^{-1} + (1-\beta)^{-1}\right),$$
where $C$ is the constant hidden in estimate \eqref{i:decay}.

\boxed{3} The desired estimate \eqref{aquad} for $A$ is obtained upon gathering the estimates derived in steps 1 and 2.
\end{proof}

\begin{remark}[Exponential Decay] \label{rem:MN} In practice, we advocate to balance the three exponentials
on the right hand side of \eqref{aquad}, thereby 
imposing
$$\blue{\pi^2/(2k)}\approx (\beta-s^+) k N\approx (1-\beta) kM.$$
Thus, given $k>0$, we set 
$$ \blue{N=\bigg\lceil \frac {\pi^2}{2(\beta-s^+) k^2}\bigg\rceil \quad \hbox{ and }\quad
M=\bigg \lceil \frac {\pi^2}{2(1-\beta) k^2}\bigg \rceil,}
$$
which leads to
$$
\begin{aligned}
 \|(A^{-\beta} - & \cQ^{-\beta}_k(A))\pi f\|_{D(A^{s+t})}\\
 &\preceq 
 \bigg[ \frac 1 {\beta-s^+} + \frac 1 {1-\beta} \bigg] \bigg[  \frac {e^{-\blue{\pi^2/(2k)}}}{\sinh(\blue{\pi^2/(2k)})}+
e^{-\blue{\pi^2/(2k)}} \bigg] \|f\|_{D(A^t)} .
\end{aligned}
$$
Note that the coefficient on the right hand side above asymptotically behaves like 
$$ \bigg[ \frac 1 {\beta-s^+} + \frac 1 {1-\beta} \bigg]  e^{-\blue{\pi^2/(2k)}} \quad \textrm{as} \quad  k\rightarrow 0.$$ 
This choice will be refer to as the ``balanced'' scheme.
Another possibility is to simply take $k^{-2} \sim N = M$, which also lead to exponential decay but, as we shall see in Section~\ref{s:numerical}, is less efficient.
\end{remark}

\def\VV{\mathbb V}
\section{The Numerical Approximation of \eqref{example_space}-\eqref{example}.}\label{app-example}
For the remainder of this paper, we focus on the example described by
\eqref{example_space}-\eqref{example}.   The coercivity and boundedness assumptions made in
the introduction guarantee that \eqref{i:form-strongly-elliptic} and
\eqref{i:form-bounded} hold with $Y=\VV$ so that  the operator
$A$ associated with \eqref{example} and its fractional powers  are defined in  
Section \ref{ss:frac pow}.

\subsection{The Finite Element Approximation}\label{ss:fem}
We now assume that $\Omega$ is a polyhedon. 
Let $\{\cT_h\}_{h>0}$ be a sequence of conforming subdivisions
made of simplices with maximal mesh size $h<1$. We further assume that $\{\cT_h\}$ are shape-regular and quasi-uniform (cf. \cite{EG})
with constants independent of $h$
and the triangulation matches the partitioning $\Gamma=\overline{\Gamma_D\cup\Gamma_N}$. 
Let $\vh$ be the space of continuous
piecewise linear finite element functions subordinate to  $\cT_h$
and let $\pi_h$ be the $L^2$-orthogonal projector onto $\vh$.
The semi-discrete approximation $u_h$ and the fully discrete approximation $u_{h,k}$
are defined as in the introduction.

  Before going further, we note that the quasi-uniformity assumption is required to guarantee the  $H^1(\Omega)$ stability of $\pi_h$, namely
\begin{equation}\label{i:l2-stable}
	\|\pi_h v\|_{H^1(\Omega)}\preceq \|v\|_{H^1(\Omega)}.
\end{equation}
The quasi-uniformity assumption can be relaxed, for instance to allow  certain grading condition on $\cT_h$; see \cite{thomee,BPS,Bank}.
The $H^1$ stability estimate \eqref{i:l2-stable} implies, by interpolation, that for $r \in [0,1]$
\begin{equation}\label{i:l2-stable_inter}
	\|\pi_h v\|_{\HH^r(\Omega)}\preceq \|v\|_{\HH^r(\Omega)}.
\end{equation}

\subsection{Elliptic regularity and finite element error
  estimates.} \label{ss:ell-reg}

We define $\HH_a^{-1}$ to be the set of bounded anti-linear functionals on
$\VV$.  As usual, $L^2(\Omega)$ can be imbedded in $\HH_a^{-1}$ by
identifying $f\in L^2(\Omega)$ with the functional
$$\langle F,\phi \rangle = (f,\phi), \Forall \phi\in \VV$$
with $\langle \cdot,\cdot \rangle$ denoting the anti-linear
functional/function pairing.
We define the spaces $\HH^r_a$ to $r\in [-1,0)$ by setting
  $$\HH^r_a=(\HH^{-1}_a,L^2(\Omega))_{r+1,2}.$$
It follows that $T:L^2(\Omega)\rightarrow D(A)$ extends to a bounded
antilinear map of $\HH^{-1}_a\rightarrow \VV$ defined by replacing
$(f,\phi)$ with $\langle F,\phi \rangle$.

The adjoint operator $T^*:L^2(\Omega)\rightarrow \VV$ is defined
analogously to $T$, i.e.,
$$A(\phi,T^*g)=(\phi,g),\Forall \phi\in \VV.$$
In this case, $T^*g$ is a linear functional and extends to  $\HH^{-1}_l$,
the set of bounded linear functionals on $\VV$.  As above, we define
  $$\HH^r_l=(\HH^{-1}_l,L^2(\Omega))_{r+1,2}.$$

The analysis of the error between $u$ and $u_h$ relies on the regularity
of $T$ and $T^*$ described in the following assumption.
\begin{assumption}\label{a:elliptic-regularity} There exists $\alpha\in
  (0,1]$ such that:
\begin{enumerate}[(a)]
 \item $T$ is an isomorphism from $\HH_a^{-1+r}(\Omega)$ to $\HH^{1+r}(\Omega)$
   for any $r\in(0,\alpha]$.
 \item $T^*$ is an isomorphism from $\HH_l^{-1+r}(\Omega)$ to $\HH^{1+r}(\Omega)$
for any $r\in(0,\alpha]$.
\end{enumerate}
\end{assumption}

Let $P_h:\VV\rightarrow \vh$ denote the elliptic projection defined by
$$A(P_hu,v)=A(u,v),\Forall v\in \vh.$$
We note that it is a consequence of
Assumption~\ref{a:elliptic-regularity} and standard finite element error
analysis arguments that for $r\in (0,\alpha)$
\beq
\|(I-P_h)u\|_{H^1(\Omega)} \preceq  h^{r_1}\|u\|_{H^{1+r_1}(\Omega)},\Forall u\in
H^{1+r_1}(\Omega)\cap \VV
\label{h1-fe}
\eeq
and 
\beq
\|(I-P_h)u\|_{\HH^{1-r_2}(\Omega)} \preceq h^{r_2}\|u\|_{H^{1}(\Omega)},\Forall u\in \VV
\label{dual-fe}
\eeq
 
From the equivalence result in \cite{kato1961} and \eqref{i:l2-stable_inter}, we deduce that the discrete norms 
\beq
\|v_h\|_{\HH^r(\Omega)}\approx \|A_h^{r/2} v_h\|, \Forall v_h\in \vh,
\label{i:discrete-equiv} 
\eeq
are equivalent for $r\in [0,1)$. This equivalence holds also for the adjoint operator $A_h^*$ and the case
$r=1$ follows from Assumption~\ref{a:elliptic-regularity} (see, \cite[Theorem 6.5]{BP16}). 

When $r>1$, we have the following lemma.
\begin{lemma}\label{l:pi-bound}
Let $r\in (1,1+\alpha]$ where is $\alpha$ the regularity index in Assumption~\ref{a:elliptic-regularity}. 
For $v\in \HH^{r}(\Omega)$, there holds
$$
	\|A_h^{r/2}\pi_h v\|\preceq \|v\|_{\HH^{r}(\Omega)} .
$$
\end{lemma}
\begin{proof}
We first show that 
\begin{equation}\label{i:riesz-stable}
	\|A_h^{r/2}P_h v\|\preceq \|v\|_{\HH^{r}(\Omega)} .
\end{equation}
In fact, 
$$
\begin{aligned}
	\|A_h^{r/2}P_h v\| &=\sup_{\theta_h\in \vh}\frac{(A_h^{r/2} P_h v,\theta_h)}{\|\theta_h\|}\\
	&=\sup_{\theta_h\in \vh}\frac{A(P_h v,(A_h^*)^{r/2-1}\theta_h)}{\|\theta_h\|}
	=\sup_{\theta_h\in \vh}\frac{A(v,(A_h^*)^{r/2-1}\theta_h)}{\|\theta_h\|} .
\end{aligned}
$$
Then we let $\phi_h = (A_h^*)^{r/2-1}\theta_h$ and apply the discrete norm equivalence 
\eqref{i:discrete-equiv} for $A_h^*$ to yield
$$
\begin{aligned}
	\|A_h^{r/2}P_h v\| &=\sup_{\phi_h\in \vh}\frac{A(v,\phi_h)}{\|(A_h^*)^{1-r/2}\phi_h\|}
	\preceq \sup_{\phi_h\in \vh}\frac{A(v,\phi_h)}{\|\phi_h\|_{\HH^{1-r/2}(\Omega)}}\\
	&= \sup_{\phi_h\in \vh}\frac{(A^{r/2}v,(A^*)^{1-r/2}\phi_h)}{\|\phi_h\|_{\HH^{1-r/2}(\Omega)}}
	\le \sup_{\phi\in \HH^{1-r/2}(\Omega)}\frac{(A^{r/2}v,(A^*)^{1-r/2}\phi)}{\|\phi\|_{\HH^{1-r/2}(\Omega)}}\\
	&\preceq \|v\|_{\HH^r(\Omega)} .
\end{aligned}
$$

We also need the inverse estimate. Note that for $v_h\in \vh$,
$$
	\|A_h^{1/2}v_h\|\approx \|v_h\|_{H^1(\Omega)}\preceq h^{-1}\|v_h\|.
$$
This implies that for $s\in[0,1]$,
\begin{equation}\label{i:inverse-esti}
	\|A_h^{s/2}v_h\|\preceq h^{-s}\|v_h\|.
\end{equation}

Now, we invoke \eqref{i:riesz-stable} together with \eqref{i:l2-stable}, \eqref{h1-fe} 
and the inverse inequality \eqref{i:inverse-esti} to conclude that
$$
\begin{aligned}
	\|A_h^{r/2}\pi_h v\|&\le \|A_h^{r/2}P_h v\| +\|A_h^{r/2}(\pi_h-P_h)v\|\\
	&\preceq \|v\|_{\HH^{r}(\Omega)}+ h^{1-r} \|(\pi_h-P_h)v\|_{H^1(\Omega)}\\
	&\preceq \|v\|_{\HH^{r}(\Omega)}+ h^{1-r} \|(I-P_h)v\|_{H^1(\Omega)}\preceq \|v\|_{\HH^r(\Omega)}.	
\end{aligned}
$$
\end{proof}
\subsection{The error between $u$ and $u_{h,k}$.} \label{ss:uuhk}
In this section, we analyze the error between $u$ and $u_{h,k}$ measured
in the norm of $\HH^r(\Omega)$, for $r\in [0,1]$.

The following theorem provides conditions on $f$ which imply exponential
convergence of the 
  error
 $\|u_h-u_{h,k}\|_{\HH^r(\Omega)}$.
 
\begin{theorem}\label{l:expex} Let  $r$ be in $[0,1]$
 and $N$
and $M$ be as in \eqref{rem:MN}. 
Then for $\beta>r/2$,
$$
  \|u_h-u_{h,k}\|_{\HH^r(\Omega)} \preceq 
  e^{-\blue{\pi^2/(2k)}} \|f\|$$
and for $\beta\le r/2$ and $f\in \HH^{r-2\beta+\epsilon}(\Omega)$ 
satisfying $r-2\beta+\epsilon\in[0,1+\alpha]$,
$$
  \|u_h-u_{h,k}\|_{\HH^r(\Omega)} \preceq 
  e^{-\blue{\pi^2/(2k)}}\|f\|_{\HH^{r-2\beta+\epsilon}(\Omega)}.
$$
\end{theorem}

\begin{proof}  
Applying  \eqref{i:discrete-equiv}  gives
$$ \|u_h-u_{h,k}\|_{\HH^r(\Omega)} \preceq \| A_h^{r/2} (u_h-u_{h,k}) \| .$$
If $\beta>r/2$, we apply Theorem~\ref{l:exp} with $s=r/2$
  and $t=0$ to obtain 
$$
\| A_h^{r/2} (u_h-u_{h,k}) \| 
\preceq  
e^{-\blue{\pi^2/(2k)}} 
\|\pi_h f\| \preceq  
e^{-\blue{\pi^2/(2k)}} 
\|f\|.
$$
Alternatively, when $\beta\le r/2$, we note that $t:=(r+\epsilon)/2-\beta$ is in $[0,(1+\alpha)/2]$
and $s:=\beta-\epsilon/2< \beta$. Now,
applying Theorem~\ref{l:exp}, Lemma~\ref{l:pi-bound}, \eqref{i:discrete-equiv} and
\eqref{i:l2-stable_inter} gives
$$
\| A_h^{r/2} (u_h-u_{h,k}) \|
\preceq  
e^{-\blue{\pi^2/(2k)}} 
\|\pi_h f\|_{D(A_h^t)} \preceq  
e^{-\blue{\pi^2/(2k)}} 
\|f\|_{\HH^{2t}(\Omega)}. 
$$
Combining the above three inequalities completes the proof of the theorem.
\end{proof}

We next discuss the error bounds for $\|u-u_h\|_{\HH^r(\Omega)}$ given in
Theorem 6.2 of \cite{BP16} with $r\in [0,1]$.     
The proof given there uses both \eqref{h1-fe} and
\eqref{dual-fe} with of $r_1= \alpha$ and $r_2=
\min\{\alpha,1-r\}$ and results in an order of convergence 
 $0<2\alpha^*=r_1+r_2$ (with possibly a logarithm of $h^{-1}$
deterioration depending on the regularity of $f$).
  
Theorem 6.2 of \cite{BP16}  shows that:
\begin{enumerate}[Case 1:]
\item [Case~$1$:]
When $r/2+\alpha^*-\beta\ge 0$ and $f$  is in $D(A^{r/2+\alpha^*-\beta})$
$$
\bal
\|u-u_h\|_{\HH^r(\Omega)}&\preceq \log(h^{-1}) h^{2\alpha^*}
\|f\|_{D(A^{r/2+\alpha^*-\beta})}\\
&\preceq 
\log(h^{-1}) h^{2\alpha^*}
\|f\|_{\HH^{r+2\alpha^*-2\beta}(\Omega)}.
\eal
$$
\item  [Case~$2$:]
When $r/2+\alpha^*-\beta\ge 0$ and $f$ is in $D(A^{
  r/2+\alpha^*-\beta+\epsilon})$ with $r+2\alpha^*-2\beta+2\epsilon\le 1+\alpha$,
$$
\|u-u_h\|_{\HH^r(\Omega)}\preceq  h^{2\alpha^*}
\|f\|_{D(A^{r/2+\alpha^*-\beta+\epsilon})}\preceq 
h^{2\alpha^*}
\|f\|_{\HH^{r+2\alpha^*-2\beta+2\epsilon}(\Omega)}.
$$
\item  [Case~$3$:]
When $r/2+\alpha^*-\beta<0 $ and $f\in L^2(\Omega)$,
$$
\|u-u_h\|_{\HH^r(\Omega)}\preceq
h^{2\alpha^*} \|f\|.
$$
\end{enumerate}

The following theorem follows from combining the above results with
Theorem~\ref{l:expex}.

\begin{theorem} \label{t:combined} For $r\in [0,1]$, $k>0$ and $N$
and $M$ be as in \eqref{rem:MN}.
Then we have
$$
	\|u-u_{h,k}\|_{\HH^r(\Omega)} \preceq  
	\left\lbrace
	\begin{array}{ll}
	(\log(h^{-1})  h^{2\alpha_*}
          +e^{-\blue{\pi^2/(2k)}}) \|f\|_{\HH^{r+2\alpha^*-2\beta}(\Omega)} &  \textrm{in}\mathrm{~Case~1},\\
           ( h^{2\alpha_*}
          +e^{-\blue{\pi^2/(2k)}})\|f\|_{\HH^{r+2\alpha^*-2\beta+2\epsilon}(\Omega)} &  \textrm{in}\mathrm{~Case~2},\\
          ( h^{2\alpha_*}
          +e^{-\blue{\pi^2/(2k)}})\|f\| &  \textrm{in}\mathrm{~Case~3}.
          \end{array}\right.
$$
\end{theorem}

\begin{proof}  
We set $C_1:= \|f\|_{\HH^{r+2\alpha^*-2\beta}(\Omega)}$, $C_2 = \|f\|_{\HH^{r+2\alpha^*-2\beta+2\epsilon}(\Omega)}$ and $C_3 = \|f\|$.
We note that
\beq
\|f\|_{\HH^s(\Omega)} \le 
\|f\|_{\HH^t(\Omega)}, \hbox { for } 0\le s\le t\le 2\hbox{ and } f\in
\HH^t(\Omega).
\label{hh-imb}
\eeq
This means that
$\|f\|\le C_j$, for $j=1,2,3$.  Thus, when $\beta>r/2$,  Theorem~\ref{l:expex} yields
$$	\|u_h-u_{h,k}\|_{\HH^r(\Omega)} \preceq C_j
e^{-\blue{\pi^2/(2k)}},\quad \hbox{ for }j=1,2,3.$$
When $\beta\le r/2$, we are in Cases~1 or 2 and Theorem~\ref{l:expex} with $\epsilon=2\alpha^*$ together with \eqref{hh-imb} imply
$$
\|u_h-u_{h,k}\|_{\HH^r(\Omega)} 
\le e^{-\blue{\pi^2/(2k)}} \|f\|_{\HH^{r-2\beta+2\alpha^*}}\le
C_je^{-\blue{\pi^2/(2k)}},
$$
for $j=1$ or $j=2$.
The desired result follows from the above estimates, the bounds for
$\|u-u_h\|_{\HH^r(\Omega)}$ above and the triangle inequality.
\end{proof}

\section{Numerical Illustration.}\label{s:numerical}
In this section, we present some numerical experiments to illustrate the error estimates derived in Section~\ref{s:SINC}.
In order for the error due to the finite element approximation to shadow the exponentially converging sinc quadrature error, we consider the following one dimensional problem:
\begin{equation}\label{e:oned-problem}
\begin{aligned}
	A^\beta u &= 1, \quad\text{ in } (0,1),\\
	u(0) &= u(1) = 0,
\end{aligned}
\end{equation}
where $A$ is an unbounded operator associated with the bilinear form 
$$A(u,v)=\int_0^1 u'v'\, dx \qquad\text{for }u,v\in H^1_0(0,1)$$
with $H^1_0(0,1):=\{v\in H^1(0,1) : v(0)=v(1)=0\}$.
In particular, we provide numerical evidence of exponential rates not explained by the previous theory in  \cite{BP16} but by Theorem~\ref{l:exp}. 
Finally, we refer to \cite{BP:13,BP16} for numerical experiments with domains $\Omega \subset \mathbb R^d$, $d>1$, and general operators $A$.

\subsection{Error from the Sinc Approximation}
We first report the sinc approximation error $\|(A_h^{-\beta}-Q_k^{-\beta}(A_h))\pi_h f\|_{\HH^r(\Omega)}$ for $r\in [0,1]$. 
To this end, we consider a subdivision $\cT_{h}$ made of uniform intervals of length $h=1/512$. The finite element approximation
of $u$ in \eqref{e:oned-problem} is given by
\begin{equation}\label{e:numerical-uh}
	u_{h} = A_h^{-\beta}\pi_h f =  \sum_{\ell=1}^{D_{h}} \lambda_{\ell,h}^{-\beta}  (1,\psi_{\ell,h})  \psi_{\ell,h},
\end{equation}
where the number of degrees of freedom $D_h=511$ and $\{\psi_{\ell,h},\lambda_{\ell,h}\}$ 
are the eigenpairs of $A_h$, i.e.
$$
	\lambda_{\ell,h}= \frac{6(1-\cos(k\pi h))}{h^2 (2+\cos(k \pi h))}
	\ \text{ and }\ 
	\psi_{\ell,h} =\sqrt{\frac{6}{2+\cos(hl\pi)}}  \sum_{k=1}^{D_{h}}\sin( h \ell k \pi) \varphi_{k,h}.
$$
Similarly, the sinc approximation of $u_h$ is given by
\begin{equation}\label{e:numerical-Uh}
	u_{h,k} = Q_k^{-\beta}(A_h)\pi_h f =\sum_{\ell=1}^{D_{h}} \cQ_k^{-\beta}(\lambda_{\ell,h}) (1,\psi_{\ell,h})  \psi_{\ell,h}.
\end{equation}

\blue{
We report the errors in the discrete operator norm (see,
\eqref{i:discrete-equiv}), namely 
\beq
 e(k,r):=\|u_h-u_{h,k}\|_{D(A_h^{r/2})} 
	= \left(\sum_{l=1}^{D_h} \lambda_{\ell,h}^r |(u_h-u_{h,k},\psi_{\ell,h})|^2\right)^{1/2}.
\label{e:log-error}
\eeq
Theorem~\ref{l:exp} guarantees that $e(k,r) \leq Ce^{-c/k}$ for some constants $c$ 
and $C$ independent of $k$ and $h$.  To illustrate this behavior, we
provide  semi-log plots of the error as a function of $1/k$ so that
$e^{-c/k}$ ends up being a straight line with slope $-c$. 
Figure~\ref{f:sinc} reports the values of $e(k,r)$ for $\beta
=0.3,0.5,0.7$ and $r=0,\beta,1$.   In this case, $M$ and $N$ are chosen
to balance the three error terms coming from the sinc quadrature
analysis (see, Remark~\ref{rem:MN}). 
The results are close to straight lines 
which is  consistent with the $Ce^{-c/k}$ error behavior of the theory.  
Since $1 \in H^{1/2-\epsilon}(\Omega)$ for any $\epsilon>0$, the
exponential rate observed for $r=1$ is not explained by the previous
theory in  \cite{BP16} but results from  Theorem~\ref{l:exp} above.   
}

\blue{To illustrate the benefit in choosing the balanced scheme, we provide
similar results in Figure~\ref{f:sinc-old} for  $M=N = 1/k^2$.
Although, both strategies yield exponential decay,
the balanced scheme is dramatically more efficient.}

\begin{figure}[hbt!]
	\begin{center}
		\begin{tabular}{lll}
			\!\!\!\!\!\!\!
			\includegraphics[scale=.24]{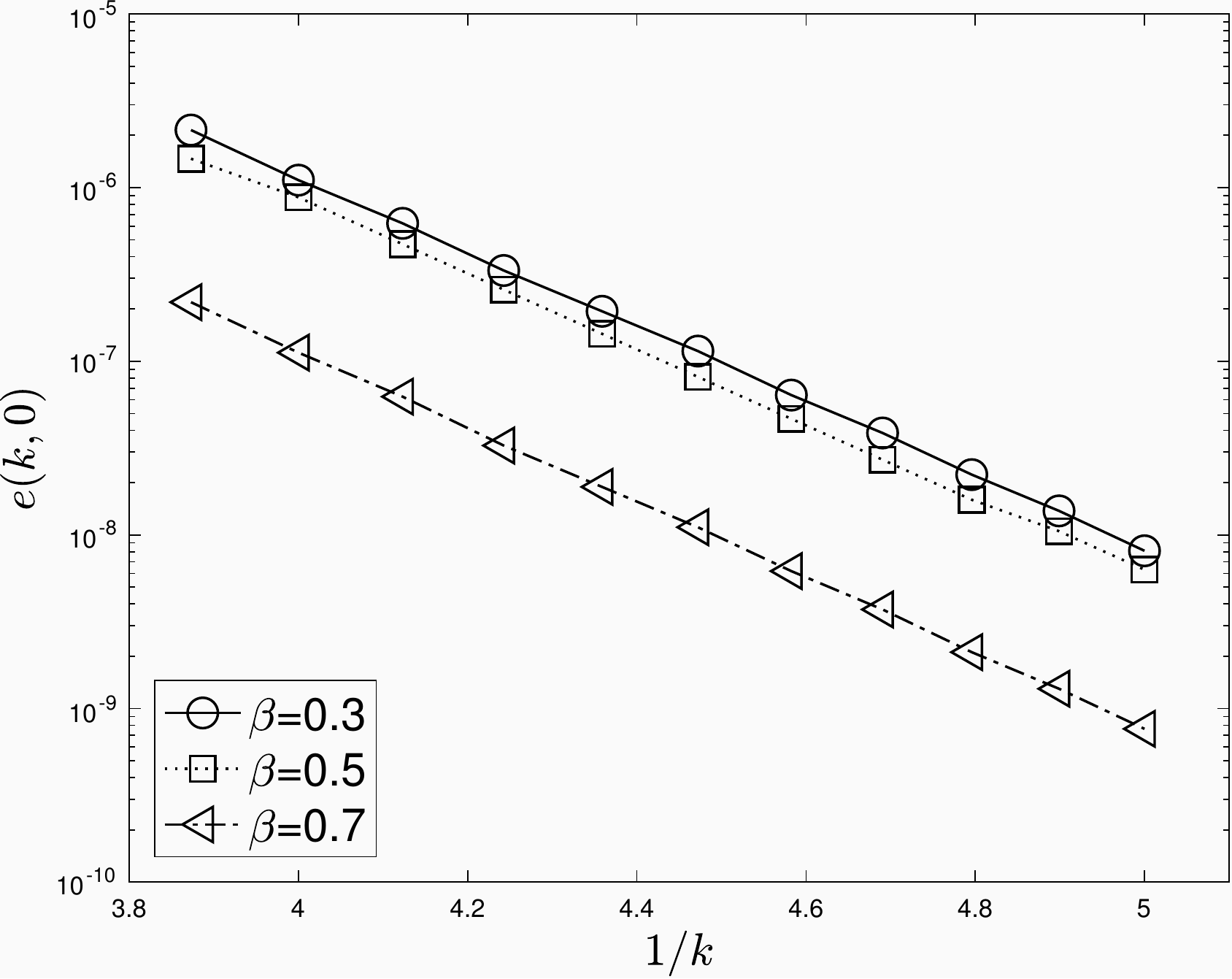} & 
			\!\!\!\!\!\!\!
			\includegraphics[scale=.24]{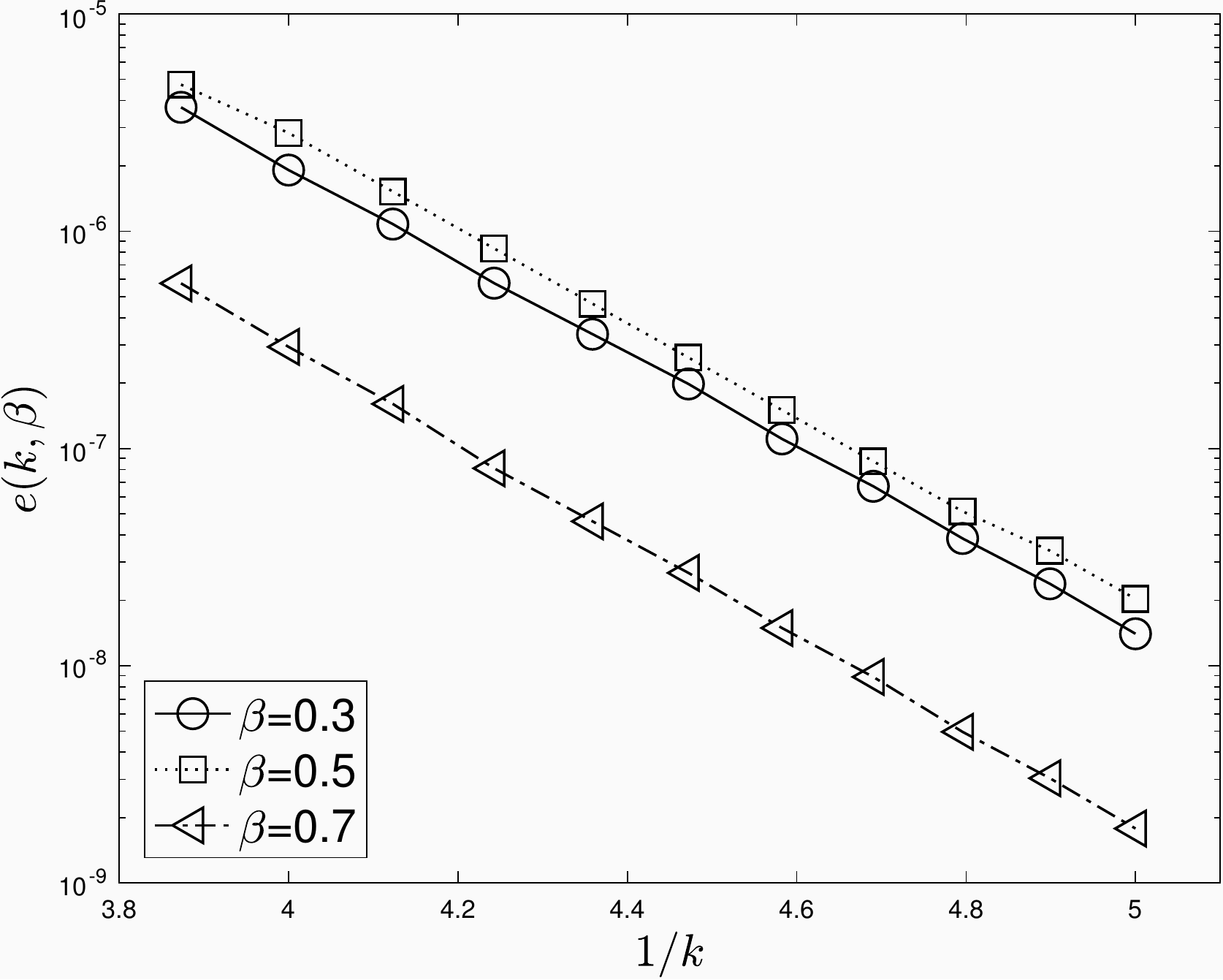}
			\!\!\!\!\!\!\!
			&\includegraphics[scale=.24]{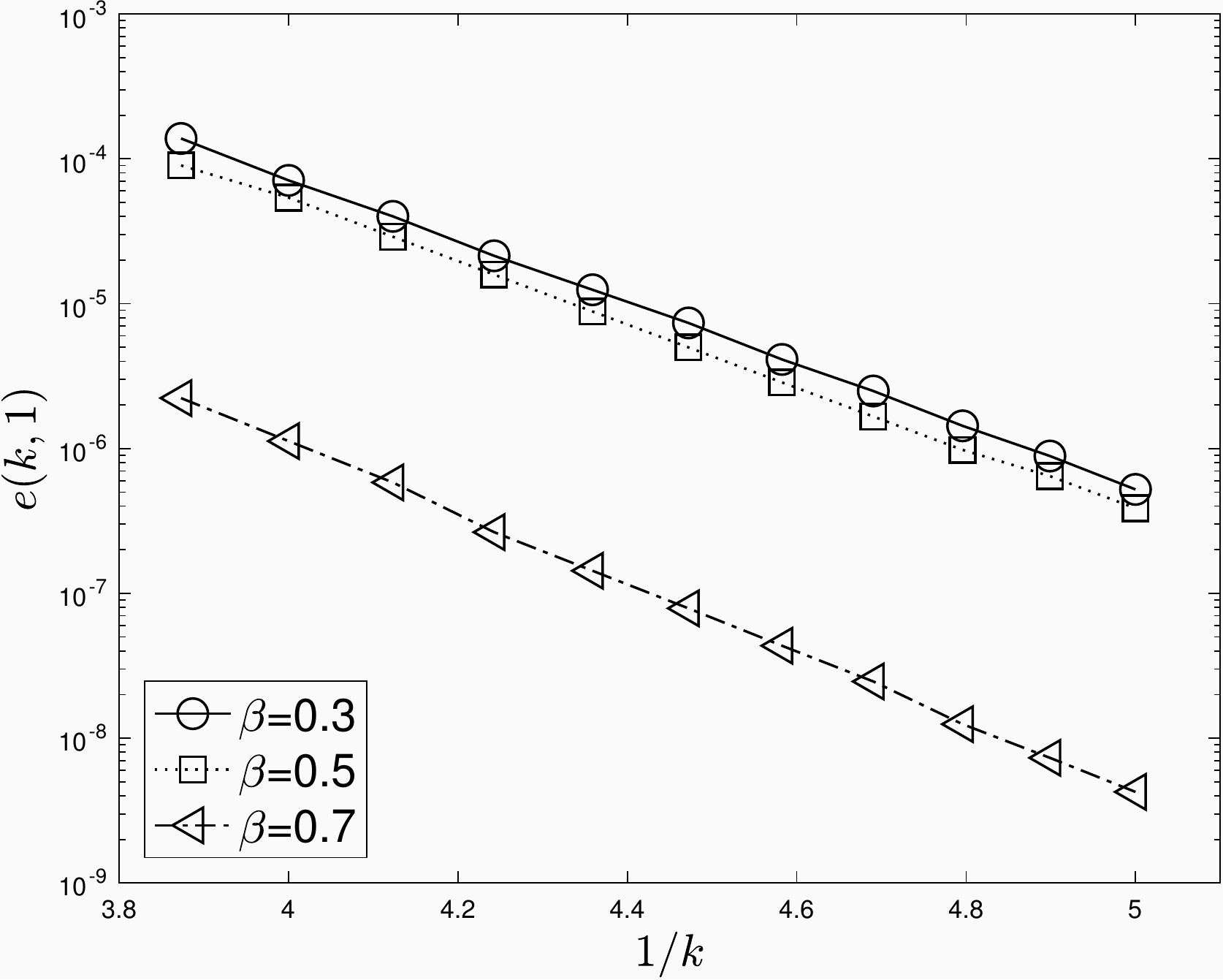}\\
    		\end{tabular}
 	\end{center}
    	\caption{\blue{Values of $e(k,r)$ defined by \eqref{e:log-error} for different sinc quadrature spacing $k$
	and $r=0$ (left), $r=\beta$ (middle) and $r=1$ (right). 
	Here $M$ and $N$ are chosen to balance the three error terms
        coming from the sinc quadrature $\HH^r(\Omega)$ for a given sinc
        quadrature spacing $k$ (see, Remark~\ref{rem:MN})}.
	}
    	\label{f:sinc}
\end{figure}

\begin{figure}[hbt!]
	\begin{center}
		\begin{tabular}{lll}
			\!\!\!\!\!\!\!
			\includegraphics[scale=.24]{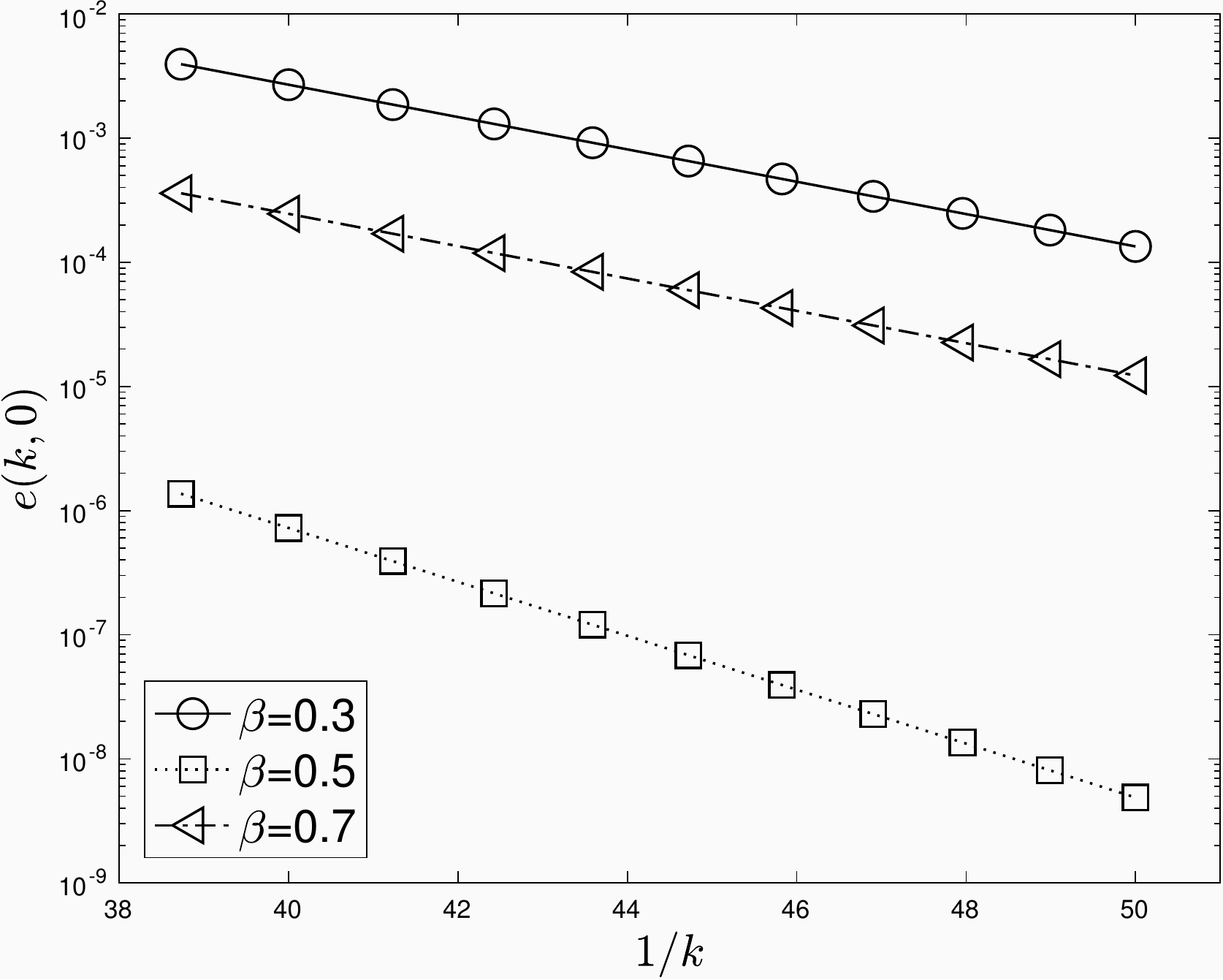} & 
			\!\!\!\!\!\!\!
			\includegraphics[scale=.24]{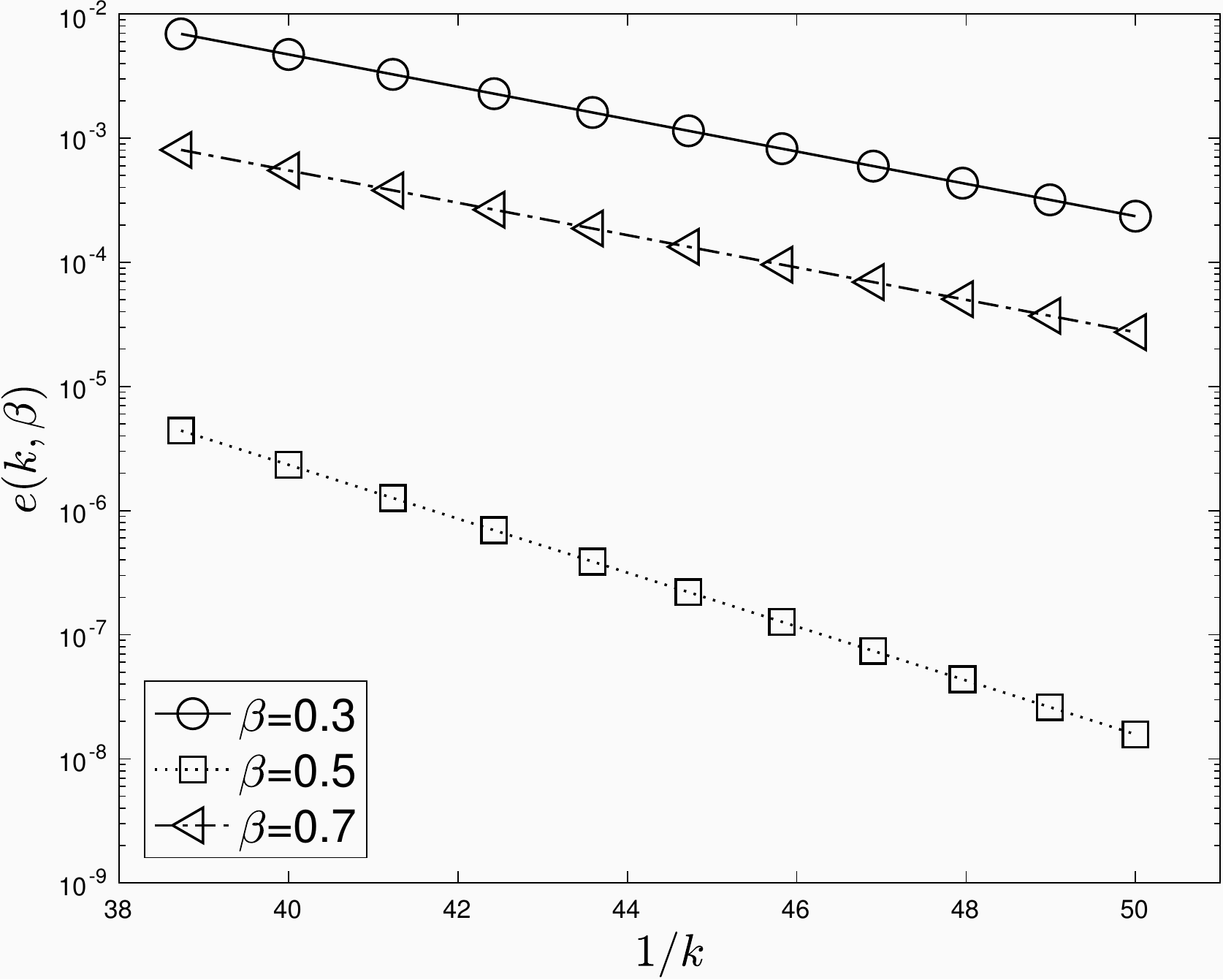}
			\!\!\!\!\!\!\!
			&\includegraphics[scale=.24]{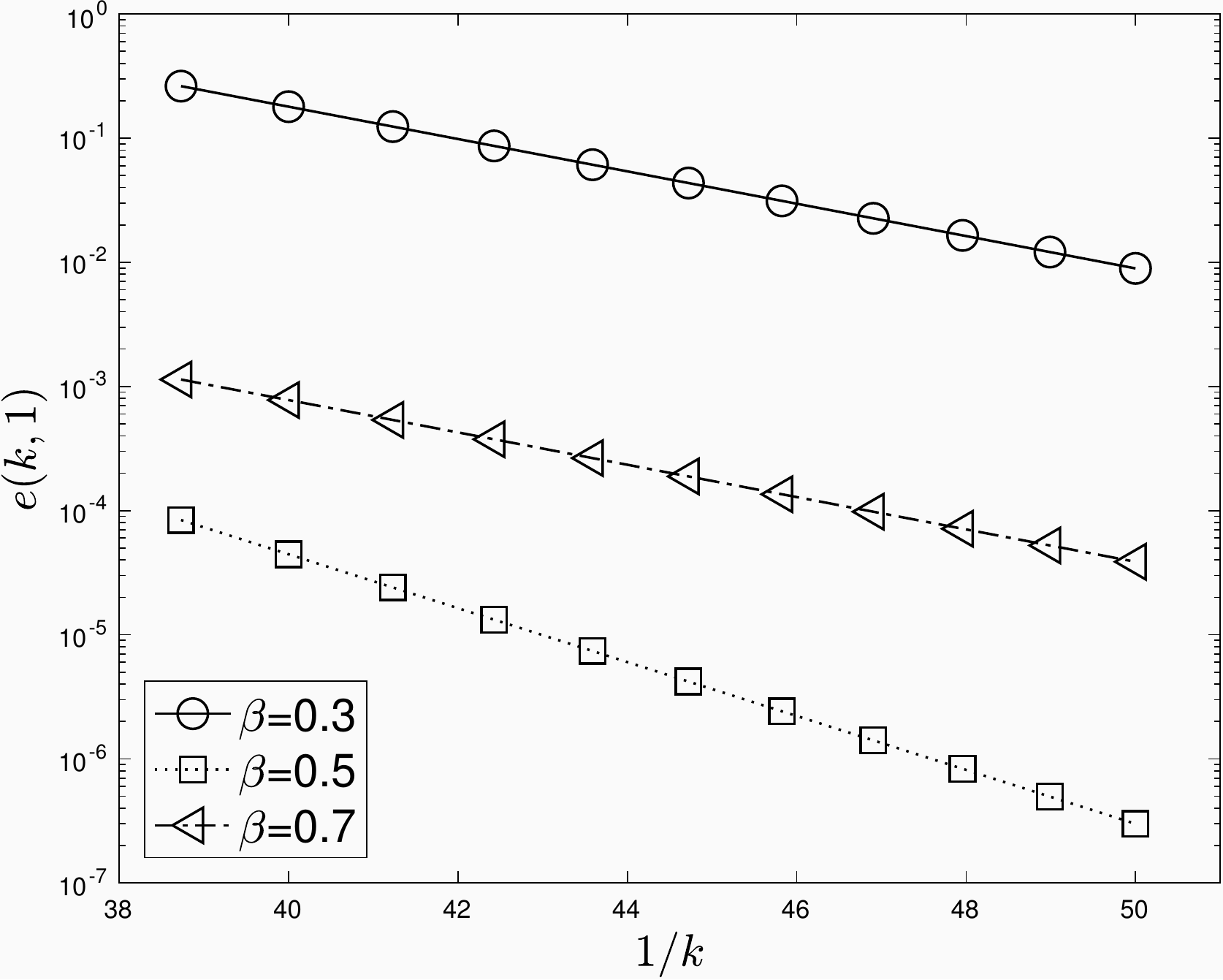}\\
    		\end{tabular}
 	\end{center}
    	\caption{\blue{Values of $e(k,r)$ defined by \eqref{e:log-error} for different sinc quadrature spacing $k$
	and $r=0$ (left), $r=\beta$ (middle) and $r=1$ (right). In
        contrast with the experiments provided in Figure~\ref{f:sinc},
        here  $M=N=1/k^2$.  }}
    	\label{f:sinc-old}
\end{figure}

\subsection{Total Error}
We compare the solution $u$ in \eqref{e:oned-problem} and its fully discrete approximation
$u_{h,k}$ given by \eqref{e:numerical-Uh}.
The sequence of meshes are obtained upon performing  successive uniform refinements of the unit interval, thereby leading to mesh sizes $h_j=2^{-j}$ for $j=3,\ldots,8$. 

Since $1 \in H^{1/2 - \epsilon}(\Omega)$ for any $\epsilon>0$, 
the predicted rate of convergence (up to a logarithmic term) is described by
\begin{equation}\label{i:err-l2}
	\|u-u_{h,k}\| \preceq h^{\max\{2,2\beta+1/2\}} 
\end{equation}
upon setting 
$$
k = \frac{1}{({8(2\beta+1/2)\log(1/h)}}
$$
and by 
\begin{equation}\label{i:err-h1}
	\|u-u_{h,k}\|_{H^1(\Omega)} \preceq  h^{\max\{2,2\beta+1/2\}-1} 
\end{equation}
with the choice
$$
k = \frac{1}{({4(2\beta-1/2)\log(1/h)})}.
$$
We note that the constant $8$ and $4$ appearing in the choices of $k$ are tuned so that the error of the sinc quadrature is already in its asymptotic regime for small $k$.
As already noted (see Figure ~\ref{f:sinc} and \ref{f:sinc-old}), this choice is not necessary for the algorithm to work. 

Also, the solution $u$ is not known exactly and is therefore approximated by truncating its eigenfunction expansion
\begin{equation}\label{e:u_exact_1d}
u \approx 2\sum_{\ell=1}^{50000}  (\pi\ell)^{-\beta} \frac{1-(-1)^\ell}{\pi \ell}  \sin(\pi \ell x).
\end{equation}
The errors between $u$ and $u_{h,k}$ are
reported in Figure~\ref{f:err} and matches the predictions of Theroem~\ref{t:combined}.

\begin{figure}[hbt!]
 \begin{center}
    \begin{tabular}{cc}
\includegraphics[scale=.31]{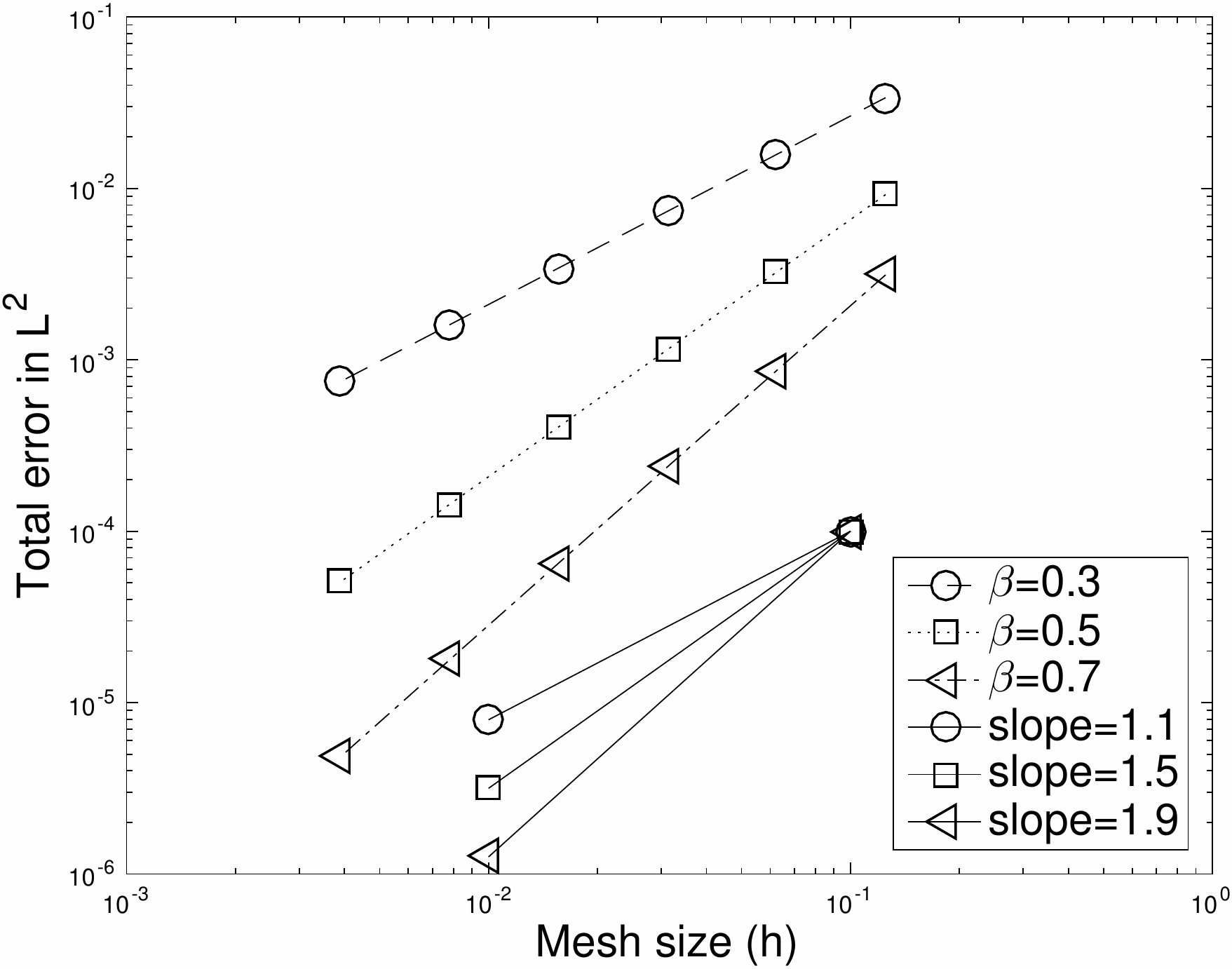} & \includegraphics[scale=.31]{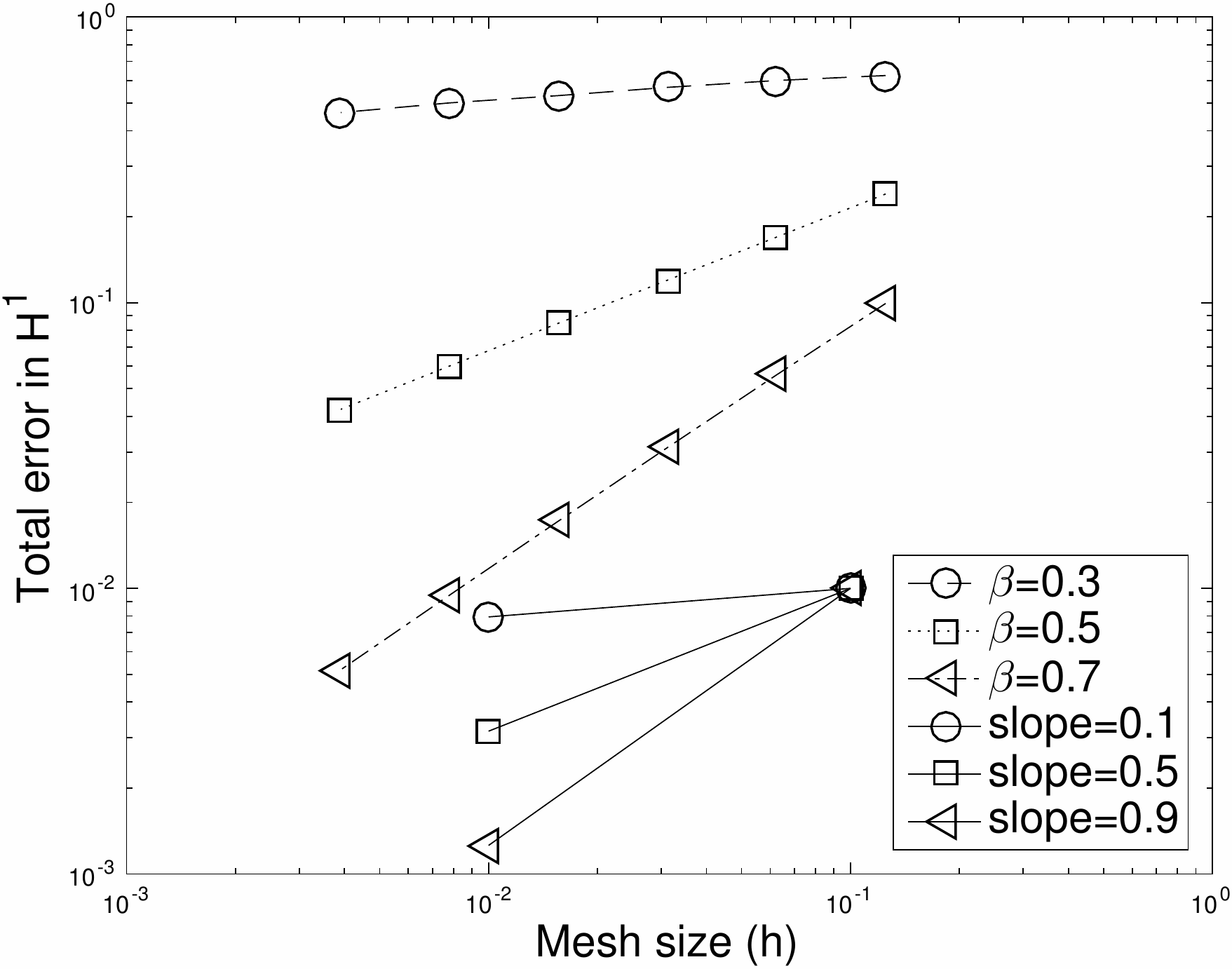}\\
    \end{tabular}
 \end{center}
    \caption{Errors between $u$ and $u_{h,k}$ in the $L^2(\Omega)$ norm (left)  and the $H^1(\Omega)$ 
    norm (right) versus the mesh size for $\beta = 0.3, 0.5, 0.7$.
    The rates of convergences predicted by \eqref{i:err-l2} and \eqref{i:err-h1} are observed numerically.}
    \label{f:err}
\end{figure}

\end{document}